\documentclass[11pt,a4paper,reqno]{amsart}

\usepackage{graphicx,pdfsync,amssymb,mathrsfs,amsfonts,amsbsy,color,esint,mathtools,tikz,mdframed}
\usetikzlibrary{positioning}

\usepackage[shortlabels]{enumitem}

\usepackage{hyperref}
\hypersetup{
colorlinks=true,
linkcolor=blue,
filecolor=blue,
urlcolor=blue,
citecolor=blue,
}

\usepackage[utf8]{inputenc}
\usepackage[T1]{fontenc}

\DeclareMathOperator{\Supp }{supp}

\DeclareMathOperator{\D}{div}

\DeclareMathOperator{\dist}{dist}

\newtheorem{theorem}{Theorem}[section]
\newtheorem{lemma}[theorem]{Lemma}
\newtheorem{proposition}[theorem]{Proposition}

\newtheorem{remark}[theorem]{Remark}

\setlist[enumerate]{itemsep=3pt}
\setlist[itemize]{itemsep=3pt}

\def \TT  {\mathbb{T}} 
\def \RR {\mathbb{R}}  
\def \NN {\mathbb{N}}  

\def \p {\partial}

\def \l {\lambda}

\def \ep {\epsilon}
\def \om {\omega}

\def \e {\mathbf{e}}

\def \ez {\mathbf{e}_z}
\def \et {\mathbf{e}_\theta}

\def \er {\mathbf{e}_r}

\numberwithin{equation}{section}

\begin{document}

\title[Illposedness of  in supercritical spaces]{Illposedness of   incompressible fluids in supercritical Sobolev spaces}

\author{Xiaoyutao Luo}

\address{Academy of Mathematics and Systems Science, Chinese Academy of Sciences, Beijing  100190, China; and Morningside Center of Mathematics, Chinese Academy of Sciences}

\email{xiaoyutao.luo@amss.ac.cn}


\date{\today}

\begin{abstract}
We prove that the 3D Euler and Navier-Stokes equations are strongly illposed in supercritical Sobolev spaces. In the inviscid case, for any $0<s<   \frac{5}{2} $, we  construct a $C^\infty_c$ initial   velocity field  with arbitrarily small $H^{s} $ norm for which the unique local-in-time smooth solution of the 3D Euler equation develops large $\dot{H}^{s}$ norm inflation almost instantaneously.  In the viscous case, the same $\dot{H}^{s}$ norm inflation occurs in the 3D Navier-Stokes equation for $0<s<   \frac{1}{2} $, where $s = \frac{1}{2}$ is   scaling critical   for this equation.
\end{abstract}

\date{\today}

\maketitle

\section{Introduction}\label{sec:intro}

In this paper, we consider two fundamental equations in incompressible fluid dynamics, the three-dimensional Euler  and  Navier-Stokes equations. They are hydrodynamics models that describe the motion of an incompressible inviscid or viscous fluid.

While the global regularity of both equations are still outstanding open questions, the aim of this paper is to show that these two equations are actually \emph{illposed} in supercritical Sobolev spaces.  

Let us first consider the inviscid case. The equation of an ideal fluid is the 3D Euler equation,
\begin{equation}\label{eq:euler}
\begin{cases}
\p_t u + u\cdot \nabla u + \nabla p = 0 &\\
\D u =0  & \\
u|_{t = 0} = u_ 0 &
\end{cases}
(t,x) \in [0,T] \times \RR^3
\end{equation}
where $ u: [0,T] \times \RR^3  \to \RR^3$ is the unknown velocity field, $ p: [0,T] \times \RR^3 \to \RR $ is the scalar pressure, and $ u_0 : \RR^3 \to \RR^3$ is the initial data.

In the literature, function spaces having the same scaling property as the Lipschitz space are called critical for \eqref{eq:euler}, such as  $ \dot{H}^{\frac{5}{2}}$. It is known that the Euler equation is locally wellposed  in $H^{s}$  if $ s>   \frac{5}{2} $ and illposed if $ s =    \frac{5}{2} $.  

Our first main result shows that the illposedness persists in all supercritical regime $0 < s <    \frac{5}{2} $.

\begin{theorem}\label{thm:Hs}
The 3D Euler equation \eqref{eq:euler} is strongly illposed in $H^{s} (\RR^3)$ for any $0< s <   \frac{5}{2}  $ in the following sense.

Let   $0< s <  \frac{5}{2} $. For any $\ep>0$, there exists a vector field $ u_0 \in C_c^\infty(\RR^3) $ such that all of the following holds.
\begin{enumerate}
\item $u_0$ is divergence-free   and
\begin{equation}\label{eq:thm_Hs_1}
| u_0 |_{H^{s}(\RR^3) } \leq \ep .
\end{equation}

\item There exists a time $0< t^* \leq  \ep   $ such that $ u \in C^\infty( [0,t^*] \times \RR^3)$ and
\begin{equation}\label{eq:thm_Hs_2}
| u   (t^*) |_{\dot{H}^{s}(\RR^3) } \geq \ep^{-1} .
\end{equation}

\end{enumerate}

\end{theorem}

Such a result is called ``norm inflation'' in the literature, where arbitrarily small data in some space $X$ can become arbitrarily large in $X$ in an arbitrarily short time. In particular, the data-to-solution map $ u_0 \mapsto u(t)$ is unbounded in $X \to C([0,T]; X )$.

Theorem \ref{thm:Hs} combined with the end-point $H^{ \frac{5}{2} }$ illposedness~\cite{MR3359050} of Bourgain-Li finds the thresholds $s=0$ and $s =  \frac{5}{2}$  sharp both from above and below: 
\begin{itemize}
\item  By the energy conservation of smooth solutions, the $ L^2$ norm is constant in time;

\item   By the classical result~\cite{MR0951744} of Kato-Ponce, local wellposedness of \eqref{eq:euler} holds in $H^s$ when $s > \frac{5}{2}$, with a bounded solution map $u_0 \mapsto u(t)$.
\end{itemize}

\begin{remark}
Our construction is  explicit and elementary, with the following additional properties.

 \begin{itemize}

\item The initial data is   axisymmetric with a large (the same order as the full velocity) swirl, but the solution remains smooth when the inflation occurs at $t = t^*$.

\item Near the critical time $t^* \ll 1$, for every $s'>0$ the norm $|u(t)|_{H^{s'}}$ has increased by a factor $\ep^{-N}$ for some large $N$ compared to the initial $|u_0|_{H^{s'}}$. However, a growth like \eqref{eq:thm_Hs_1} can only by seen for $s'\geq s$ since the initial data $ |u_0|_{H^{s'}} \ll \ep^{ N}  $ if $s' <s$.

\end{itemize}
\end{remark}

Next, we introduce the second result. This time we consider the equation of a viscous fluid, the 3D Navier-Stokes equation,
\begin{equation}\label{eq:NS}
\begin{cases}
\p_t u -\Delta u + u\cdot \nabla u + \nabla p = 0 &\\
\D u =0  & \\
u|_{t = 0} = u_ 0 &
\end{cases}
(t,x) \in [0,T] \times \RR^3
\end{equation}
where the added dissipation represents the internal friction in a fluid.  Compared to the inviscid case, \eqref{eq:NS} enjoys the following scaling invariance
\begin{equation}\label{eq:intro_scaling_NS}
\begin{cases}
u(t,x) \mapsto u_\l (t,x) : = \l u(\l^2 t,\l x ) &\\
u_0 ( x) \mapsto u_{0, \l } ( x) : = \l u_0 ( \l x ) .&
\end{cases}
\end{equation}

A Banach space $X$ is called critical if its norm $ |\cdot |_X$ is invariant under such a scaling. In particular, the Sobolev space $\dot H^s $ is critical when $s  = \frac{1}{2}$. Roughly speaking, \eqref{eq:NS} is locally wellposed in critical or subcritical spaces, modulo technical end-points cases. On the other hand, global existence (but not uniqueness) is known in $L^2$  for Leray-Hopf weak solutions~\cite{MR1555394,doi:10.1002/mana.3210040121}, where $L^2$ is supercritical with respect to the scaling \eqref{eq:intro_scaling_NS}.

It has been an open problem whether \eqref{eq:NS} is locally wellposed in $ H^s $ when $s< \frac{1}{2}$. Our result below provides a negative answer to this question. Very surprisingly, our method of proving Theorem \ref{thm:Hs} is robust enough that only minor modification is needed to handle the vicious case.

\begin{theorem}\label{thm:Hs_NS}
The 3D Navier-Stokes equation \eqref{eq:NS} is strongly illposed in $H^{s} (\RR^3)$ for any $0< s <  \frac{1}{2}  $  in the following sense.

Let   $0< s <  \frac{1}{2} $. For any $\ep>0$, there exists a vector field $ u_0 \in C_c^\infty(\RR^3) $ such that all of the following holds.
\begin{enumerate}
\item $u_0$ is divergence-free   and
\begin{equation}\label{eq:thm_Hs_NS_1}
| u_0 |_{H^{s}(\RR^3) } \leq \ep .
\end{equation}

\item There exists a time $0< t^* \leq  \ep   $ such that $ u \in C^\infty( [0,t^*] \times \RR^3)$ and
\begin{equation}\label{eq:thm_Hs_NS_2}
| u   (t^*) |_{\dot{H}^{s}(\RR^3) } \geq \ep^{-1} .
\end{equation}

\end{enumerate}

\end{theorem}

As in the inviscid case, both exponents $s = 0$ and $s = \frac{1}{2}$ are sharp:
\begin{itemize}

\item $L^2$ norm can only decrease due to the energy balance $ \frac{1}{2} |u(t)|_{L^2}^2 + \int_0^t |\nabla u (\tau)|_{L^2}^2 \,d\tau  = \frac{1}{2}| u_0 |_{L^2}^2 $

\item The 3D Navier-Stokes equation  is locally wellposed in $H^{s}$ if $s \geq \frac{1}{2}$.
\end{itemize}

We discuss our results and related background in more detail below.

\subsection{Background: the Euler case}

Given the extensiveness of the wellposedness theory for the Euler equations, we focus on the modern development and the illposedness perspective and refer interested readers to \cite{MR1867882,MR2768550} for historical accounts.

In the seminal work \cite{MR0951744}, Kato and Ponce established wellposedness in Sobolev spaces $W^{s,p}(\RR^d)$ for   dimensions $d \geq 2$, with  $ s > \frac{d}{p} + 1$  and $1 < p< \infty$.  The borderline $ s = \frac{d}{p} + 1$ is considered critical for the Euler equations as the embedding of $W^{s,p}$ into the Lipschitz class just fails. Further results in Besov spaces can be found in \cite{MR1664597,MR1717576,MR2097579,MR2072064}.

Whether the Euler equation is ill-posed in critical Sobolev spaces $W^{s,p} $ with $ s = \frac{d}{p} + 1$ was a long-standing problem. The first illposedness result dates back to \cite{MR0877643}, where DiPerna and Lions demonstrated norm inflation in $W^{1,p}(\TT^3)$ for $p<\infty$ using $2\frac{1}{2}$-dimensional shear flows examples. This construction was revisited in \cite{MR2610558} in a $C^\alpha(\TT^3)$ setting.

In the groundbreaking work \cite{MR3359050}, Bourgain and Li demonstrated strong illposedness of norm inflation and nonexistence when   $ s =  \frac{d}{p} + 1$, $p<\infty$. The work \cite{MR3359050} sparked a series of research activities on strong illposedness of the Euler equation in critical spaces \cite{MR3320889,MR4065655,MR3625192,MR3451386,MR4300224}. Remarkably, some observations in illposedness were used in Elgindi's ingenious singularity formation for   $C^{1,\alpha}$ solutions of 3D Euler equations~\cite{MR4334974}.  For recent developments in singularity formation, we refer readers to \cite{2210.07191,2305.05660}.

In addition to critical Sobolev spaces,  the Euler equation is also illposed in other spaces. We mention $C^k$ illposedness of Bourgain and Li \cite{MR3320889} and   Elgindi-Masmoudi~\cite{MR4065655}. Other types of mild illposedness were considered in early works  \cite{MR2566571,MR2606636}.

It is worth mentioning that the nonuniqueness of weak solutions constitutes another form of illposedness of the Cauchy problem. The convex integration technique, pioneered by De Lellis and Székelyhidi Jr. in the seminal work \cite{MR2600877,MR3090182}, has been used to construct weak solutions with ``unconventional'' behaviors in fluid dynamics~\cite{MR3866888,MR3898708,MR4198715,MR4462623,MR4601999}. However, these solutions exhibit very different characteristics from those aforementioned. See surveys \cite{MR3619726,MR4188806}.

The flourishing of illposedness results in critical Sobolev spaces raises the intriguing question of whether such phenomena are even possible in the supercritical regime, say $W^{s,p}$ for $ s <  \frac{d}{p} + 1$. For a long time, very little was known there, save for the $2\frac{1}{2}$-dimensional shear flows examples of DiPerna-Lions~\cite{MR0877643}, whose geometry is not sufficient to reach the full range of supercritical regime and also restricted to the periodic domain $\TT^3$ (as solutions would otherwise have infinite energy in $\RR^3$).

Very recently in \cite{2210.17458},  C\'ordoba, Mart\'inez-Zoroa, and Ożański constructed a class of global unique solutions of the 2D Euler equation for $1<s<2 $. Initially in $H^{s}$, these solutions leave $H^{s - \ep}$ for some $\ep(s)>0$ for all $t>0$, exhibiting a sudden loss in their Sobolev regularity. See also~\cite{MR4309824} of Jeong for losing smoothness continuously.  The methods in \cite{MR4309824,2210.17458} heavily rely on the 2D transport structures (the Yudovich theory), and it is not clear whether they can be applied to the 3D Euler equations.

From the preceding discussion, it is evident that there has not been a supercritical Sobolev illposedness of the 3D Euler equation in $\RR^3$ thus far. Theorem \ref{thm:Hs} not only is the first instance of such a finding it also attains the sharp end-point case of $L^2$, where norm inflation is impossible by the conservation of energy.

 \subsection{Background: the Navier-Stokes case}

Shifting our focus to the viscous case, we discuss Theorem \ref{thm:Hs_NS} within the context of recent advancements for the Navier-Stokes equations.

By scaling \eqref{eq:intro_scaling_NS}, the following spaces are critical
\begin{equation}
\dot{H}^{\frac{1}{2}} \subset L^3 \subset \dot B^{-1 + \frac{3}{p}}_{p, \infty } \,\,, p<\infty   \subset  BMO^{-1} \subset  \dot B^{-1}_{\infty, \infty } .
\end{equation}
Wellposedness in $H^{\frac{1}{2}}$ dates back to the seminar work of Fujita-Kato \cite{MR166499}. Subsequent developments in $L^3$  \cite{MR760047,MR1617394} and in Besov spaces \cite{Cannone1993-1994} culminated with Koch-Tataru's $BMO^{-1}$ theorem \cite{MR1808843}. The wellposedness theory of \eqref{eq:NS} has been an enormous field of research,  too extensive to be fully encapsulated here, and we direct the interested readers to \cite{MR3469428,MR2768550}.

There are few illposedness results for the Navier-Stokes equations compared to the inviscid case. The pioneering work of Bourgain-Pavlovi\'{c} \cite{MR2473255} proved the norm inflation in $\dot{B}^{-1}_{\infty, \infty}$, the largest critical space. This leads to subsequent developments \cite{MR2601621,MR3276597}. See \cite{MR2566571,MR4007200}  for mild illposedness in Besov spaces.

Remarkably, these the norm inflation results near $B^{-1}_{\infty, \infty}$ with integrability index $p=\infty $ are the only ones possible--- it is known that the solution map is continuous for small data in $B^{-1 + 3/p }_{p, \infty}$, see for instance~\cite[Theorems 5.27]{MR2768550}. It explains the scarcity of illposedness result in the viscous case when compared to the  Euler equations---the later equations were illposed on the whole line of $ s = \frac{d}{p} + 1$.

By comparison, our Theorem \ref{thm:Hs_NS} marks the first illposedness result of 3D Navier-Stokes in a Sobolev space, regardless of the criticality regime. Our construction shows a distinct transition for  initial data in $H^{s}$ with $s \geq \frac{1}{2}$ or $s < \frac{1}{2}$.

On the technical side, our approach differs significantly from \cite{MR2473255}, which is based on the second iterate of the mild formulation and the backward energy cascade on the Fourier side. Instead, our method hinges on the first-order approximation by the transport dynamics, as we elucidate in the subsequent discussion.

\subsection{Main ideas}

Our constructions here are inspired by   \cite{2210.17458} and our early work~\cite{MR4462623,MR4462623} and build upon two key observations.

Firstly, non-Lipschitz vector fields can induce a loss of regularity in linear transport equations, as demonstrated in~\cite{MR3933614,MR4377866,MR4381138,MR4430388}. Secondly,  anisotropy can bring desired smallness in the system, rooted in recent constructions in convex integration \cite{MR3898708,MR4462623}.

These two observations prompt the design of an anisotropic vector field for the systems \eqref{eq:euler} or \eqref{eq:NS} such that the leading order dynamics solely constitute a linear transport. The upshot is that in a suitable anisotropic setting, we can treat all other nonlinear interactions (and the dissipation of \eqref{eq:NS}) as errors, at least for a brief duration. The main task is to show this timescale is adequate for the non-Lipschitz transport to produce norm inflation.

A major departure from previous illposedness arguments for Euler equations is that we work in the  \emph{velocity} formulation instead of the vorticity formulation where $\om $ is transported by the flow in 2D (or $\frac{\om}{r}$ for 3D axisymmetric Euler \cite{MR3359050,MR3320889}). The velocity formulation appears to be necessary to reach the threshold $L^2$.

\subsection{Outline of construction}

We now explain the specifics of the construction. In a nutshell, we use a vortex ring  to  transport a toroidal rotational flow   in cylindrical coordinates. All of them are supported in a very small toroidal domain with the diameter of the cross-section much smaller than the distance to the origin.

 In this anisotropic setup the 3D axisymmetric Euler is approximated by the stationary 2D Euler  of the vortex ring plus a linear transport of the swirl component.  The swirl component exhibits a large linear growth of  oscillations, induced by the toroidal rotation of the vortex ring. We then encode these dynamics into an approximate solution $\overline{u}$ that is divergence-free, supported in a small ring, and solves the 3D Euler equation with a small error, due to its inherent anisotropy.  See  Proposition \ref{prop:approximate} for details.

The approximate solution experiences rapid norm inflation owing to its non-Lipschitz transport dynamics. For $s$ close to $\frac{5}{2}$ in the inviscid case or $\frac{1}{2}$ in the viscous case, nearly full 3-dimensional concentration is vital. Otherwise, if the initial data were situated in the sub-critical regime, no instantaneous norm inflation would occur. 

To close the argument, we need to show the following two key points:
\begin{itemize}
\item First, the approximation $\overline{u}$ remains proximate  (in $\dot H^s$) to the exact solution $u$ until the onset of inflation;  

\item Second, the exact solution $u$ with the same initial data does not blow up in  $\dot H^{s'}$ for   $s' > s$.
\end{itemize}

In general, maintaining control over a solution in the supercritical regime is very challenging. For both the viscous and inviscid cases, existing wellpossedness theories~\cite{MR0951744,MR763762} suggest a timescale $ T \sim |   u_0 |_{W^{1+ ,\infty}}^{-1}$ within which  a smooth solution can be controlled. However, in our setting, this timescale is too short, as the approximate solution $ \overline{u} $ has not grown sufficiently within that duration. 

The remedy to this issue is still anisotropy in the approximate solution $\overline{u}$. Thanks to the smallness of the source error term, $\overline{u}$ approximates the exact solution $u$ very well on a timescale much larger than $|\nabla  u_0 |_{L^\infty}^{-1}$. This timescale $t^{*}$ defined in \eqref{eq:def_critical_t*} is the onset of norm inflation, during which one can show $|u-\overline{u} |_{L^2} $ is small by standard energy estimates. Then we use a bootstrap argument to transfer the smallness of $|u-\overline{u} |_{L^2}$  to higher Sobolev norms. This bootstrap argument allows us to not only   rule out a potential blowup of $u$ but also  obtain proximity of $u$ to $ \overline{u}$ in $\dot H^s$  before the onset of norm inflation. We refer further technical details to Section \ref{sec:no_blowup_proof}.

To handle the viscous case, a scaling argument suggests that the critical time for the norm inflation must be shorter than the dissipation timescale (the inverse of two spatial derivatives). This condition is compatible with our construction, provided that $s< \frac{1}{2}$. In this setup, we treat the dissipation term $ -\Delta u$ as an error, and hence the leading order dynamics of the Navier-Stokes solution is governed by the Euler dynamics.

\subsection{Further comments}

In the inviscid case, one should be able to modify the construction to show the norm inflation in supercritical spaces  $   {W^{k,p} }  $   for  $0< k < -1 + \frac{3}{p}   $. While we do not see a clear obstruction in extensions beyond the $H^s$ space other than technical intricacies, we choose to adhere to the $H^s$ space for the sake of simplicity and clarity.

Our argument does not imply the nonexistence of $L^\infty_t H^s$ solutions for such initial data $u_0$. Achieving sufficient control in this supercritical regime  poses significant challenges, given that the uniqueness in a supercritical regimes is generally not expected for both inviscid~\cite{MR4649134} and viscous cases \cite{MR4462623,2112.03116}

In the viscous case, although the wellposedness theory may not distinguish significantly between $L^p$ and $H^s$ cases, there are drastic differences between their norm inflation examples. Growth in $H^s$ can sorely arise from the development of oscillations as shown in this paper. In contrast, generating growth in $L^p$ necessitates the generation of large spatial concentration, a task that appears to be much more challenging for incompressible fluids.

\subsection*{Organization of the paper}
The rest of the paper is organized as follows.

After some preliminaries in Section \ref{sec:pre}, we give the construction of the initial data and the approximate solution in Section 3 for the 3D Euler case and prove   Theorem \ref{thm:Hs}  in Section \ref{sec:no_blowup_proof}. The viscous case is almost identical, so in Section \ref{sec:NS_proof} we only sketch the necessary changes needed to prove Theorem \ref{thm:Hs_NS}.

\section{Preliminaries}\label{sec:pre}

\subsection{Notations}

All norms in this paper are taken on $\RR^3$, so we use $|\cdot |_{L^p } |\cdot |_{H^s } $ and $ |\cdot |_{W^{k,p} }$ for brevity. For a vector or tensor value function $f$, $|f|$ denotes its modulus, the square root of the sum of squares of each component.

We recall the following definition of Sobolev spaces $  {  W^{s, p }}  $ for real $s \in \RR $ and $1 <  p < \infty $
\begin{equation}
|f  |_{  W^{s, p }}  =  | J^s f |_{L^p }
\end{equation}
where $J^s$ is the Bessel potential given by the Fourier multiplies $ \widehat{J^s} = (1 + |\xi|^2)^{ \frac{s}{2}}$.

When $s \in \NN$ is an integer,
\begin{equation}
|f |_{   W^{s, p }  }  \simeq   \sum_{0 \leq  k \leq s} |\nabla^k f|_{L^p  }
\end{equation}
and we have the unusual convention $H^{s} : =  W^{s, 2 }$. 

We also use the homogeneous Sobolev norm
\begin{equation}
|f |_{\dot H^s  }  =  \Big(\int_{\RR^3} |\xi|^{2s} | \widehat{f} (\xi) |^2 \, d\xi \Big)^\frac{1}{2}
\end{equation}
where $\widehat{f}$ denotes the Fourier transform in $\RR^3$.

For two quantities $X,Y$, we write $X \lesssim Y$ if $X \leq C Y$ holds for some constant $C >0$, and similarly $X \gtrsim Y$ if $X \geq C Y$, and $X \sim Y$ means $X \lesssim Y$ and $X \gtrsim Y $ at the same time.

In addition, $X \lesssim_{a,b,c,\dots} Y$ means $X \leq C_{a,b,c\dots} Y$ for a constant $C_{a,b,c\dots}$ depending on parameters $a,b,c\dots$.

\subsection{Tools}

Our main tool is the following result, proved by Kato and Ponce in \cite[Propostion 4.2]{MR0951744}.

\begin{proposition}\label{prop:kato_ponce}
Let $d\geq2$ and $u$ be a smooth solution of \eqref{eq:euler} or \eqref{eq:NS} on $[0,t_0]$ for some $t_0>0$ such that $ |\nabla u   |_{L^\infty( [0,t_0] ;L^\infty) } \leq M$ for some constant $M \geq 1$.

Then for any real number $k \geq 0$ and $1 < p< \infty$,
\begin{equation}\label{eq:prop_kato_ponce}
|u(t)|_{ {W}^{k,p}  }  \leq    |u_0 |_{  {W}^{k,p}  } e^{C_{k,p} M t } \quad \text{for all $t \in [0, t_0 ]$}
\end{equation}
for universal constants $C_{k,p} $ independent of $M$, $ t_0$ and $t$.
\end{proposition}

Note that the failure of \eqref{eq:prop_kato_ponce} at $p= \infty$ is not an artifact---the equation is illposed at endpoint integer $C^k$ spaces $k \geq 1 $~\cite{MR3320889,MR4065655}. Such a failure introduces some technical issues in our estimates in Lemma \ref{lemma:bootstrap_Hk}, which we handled using the smallness from anisotropy.

We will also use the following standard Gronwall's inequality. 
\begin{lemma}
Let $\eta(t)$ be a non-negative, absolutely continuous function on $[0,T]$ satisfying
\begin{equation}
\frac{d}{d  t} \eta(t) \leq \alpha(t) \eta(t) + \beta(t)
\end{equation}
for nonnegative, summable functions $\alpha(t) $ and  $\beta(t)$ on $[0,T]$. Then
\begin{equation}
\eta(t)  \leq e^{\int_0^t \alpha (\tau ) \, d\tau } \Big( \eta(0) + \int_0^t \beta(\tau)\, d\tau \Big).
\end{equation}
\end{lemma}

\section{The approximate solution}\label{sec:overline_u}

In this section, we initiate the proof of Theorem \ref{thm:Hs} and construct the initial data $ u_0$ and an approximate solution $\overline{u}$ of the 3D Euler equation. The construction relies on  two parameters  whose values we fixed in Section \ref{sec:no_blowup_proof} depending on the given $\ep>0$.

The main goal here is to show that the approximate solution $\overline{u}$ does develop the desired $\dot{H}^s$-norm inflation at the critical time $t^* \ll \ep $ defined in \eqref{eq:def_critical_t*} below. In the next section, we will show that the approximate solution $\overline{u}$ also stays close to the exact solution $u$ at least until $t^*$.

\subsection{Auxiliary coordinates}

The construction will be given in the cylindrical coordinate centered at the origin with $(x_1, x_2 ,x_3) \mapsto (\theta,r,z )$ defined by
\begin{equation}
x_1 = r \cos \theta , \quad x_2 = r \sin \theta ,\quad x_3 =z 
\end{equation}
and the associated orthonormal frame
\begin{equation}
\et = (   - \sin \theta ,  \cos \theta, 0 )  \quad \er = (   \cos \theta  ,  \sin \theta , 0) \quad  \ez = (0,0,1) .
\end{equation}

For any vector field $v:\RR^3\to  \RR^3$, we denote its cylindrical components by $ v = v_\theta \et  + v_r \er +  v_z \ez $. We say a function $f :\RR^3\to  \RR $ is axisymmetric if it does not depend on $\theta$, and similarly for a vector-valued function $v :\RR^3\to  \RR^3$  if its components $ v_\theta, v_r ,v_z $ are axisymmetric.

Throughout the paper, for $k\in \NN$, $ \nabla^k $ refers to the full gradient in $\RR^3$. By an induction argument and the identity $\nabla f = \p_r f \er + \p_z f \ez $,  for any axisymmetric $f: \RR^3 \to \RR $ we have the following standard point-wise bounds for $k\in \NN$
\begin{equation}\label{eq:diff_r_z}
|\nabla^k f (x)|\lesssim_k \sum_{ 0\leq i  \leq  k}   | \p_r^{ i} \p_z^{k-i} f (x)|.
\end{equation}

We also work with  a shifted polar coordinate in the $rz$-plane (see Figure \ref{fig:schematic}) centered at the point $(r,z) = (\nu^{-1}  ,0  ) $ for some   $\nu\geq 1 $
\begin{equation}\label{eq:def_rho_varphi}
r = \nu^{-1}  + \rho \cos(\varphi ), \quad z = \rho \sin(\varphi ).
\end{equation}
Here $\nu$ is a large parameter whose value we will fix in Section \ref{sec:no_blowup_proof} depends on $\ep>0$, see also \eqref{eq:def_nu_lambda} below.

We also note the following useful point-wise bounds: if $f :\RR^3 \to \RR  $ is axisymmetric, then for $ k \in \NN$
\begin{equation}\label{eq:diff_rho_varphi}
|\nabla^k f (x)|\lesssim_k \sum_{ 0\leq i ,j \leq  k} \frac{1}{\rho^{k - i - j }} | \p_\rho^{ i} \p_\varphi^j f (x)| \quad \text{when $|\rho| >0$}
\end{equation}
which can be proved by passing first to \eqref{eq:diff_r_z} and then use induction in the $rz$-plane.

We use the convention that if a function $f$ is defined by variables $r,z,\rho,\varphi $, we use the same letter $f$ to indicate its Euclidean counterpart $\RR^3\to \RR$ defined implicitly by these variables and vice versa.

It is worth emphasizing that all the norms $W^{k,p}$ or $H^s$ appearing below are taken in the original Euclidean variable of $\RR^3$, and we only use variables $(\theta,r, z)$ and $(\rho , \varphi )$ to simply the notations.

\subsection{The setup}\label{subsec:setup}

The initial data $u_0$ will be obtained from translating and rescaling some fixed basic building blocks by some large frequency parameters.

Let $\mu \geq 1 $ be a large parameter and we set
\begin{equation}\label{eq:def_nu_lambda}
\begin{cases}
b  =  \frac{  \frac{5}{2}  - s }{ 100 } >0 &\\
 \nu   = \mu^{1- b }   .&
\end{cases}
\end{equation}

The major frequency parameter $\mu  \gg \nu $ represents the magnitude of a   derivative, while   $\nu^{-1}$ is the scale of the distance to the origin.

The value of $\mu$ is fixed till the very end of the proof depending on $\ep>0$ and other universal constants. Effectively, in the construction we have the freedom to tweak two parameters $\ep > 0$ small and $\mu \geq 1$ large.

For the 3D Euler equation, the choice of $b>0$ in \eqref{eq:def_nu_lambda} ensures   $  t^* \to 0$ in \eqref{eq:def_critical_t*} for $s< \frac{5}{2}$. All other results in this section hold for any $b>0$. In Section \ref{sec:NS_proof}, we will choose a different $b (s) >0$ for the viscous case.

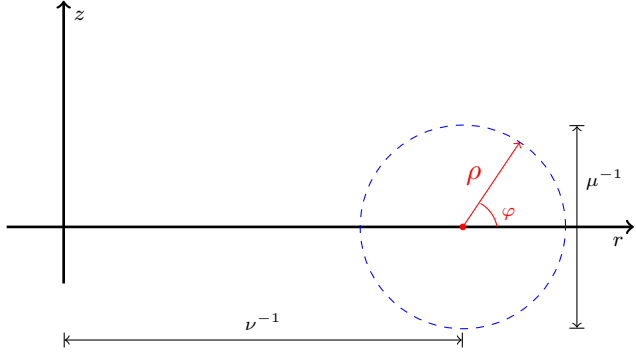
\begin{figure}[ht]
 \centering
\begin{tikzpicture}[scale=1.5]

\draw[->][line width=1]  (-0.5,-0.5) -- (-0.5,2);
\draw  (-0.5,2) node[anchor=north west] {\scriptsize  $z$};
\draw[->][line width=1] (-1,0) -- (4.5,0);
 \draw (4.5,0)  node[anchor=north east] {\scriptsize $r$};

\draw[->][color= red  ] (3,0) -- (3.5,0.75);
\draw  (3.25,0.45) node[anchor= east] {  $\color{red} \rho$  };

\filldraw[color= red  ]  (3,0)  circle (0.7pt); 

 \draw[|<->|] (-0.5,-1)--++ (3.5,0) node[pos=0.5,above]{\tiny $\nu^{-1}$};

 \draw[|<->|] (4,-0.9)--++ (0,1.8); 

\draw (4.5, 0.6) node[anchor=north east]{\tiny $\mu^{-1}$};

 \draw[blue  ,dashed] (3,0) circle (0.9cm);

\draw[color= red  ] ( 3.3,0) arc(10:60:0.3) node[pos=0.5,right]{\tiny $\varphi$};;

\end{tikzpicture}
\caption{Schematic of the shifted polar coordinate $(\rho , \varphi)$ in $rz$-plane. The scale of the profiles is $\mu^{-1}$ that is much smaller than $\nu^{-1}$, the distance to the origin. }
\label{fig:schematic}
\end{figure}

Using the large parameters $\mu,\nu  \geq 1 $ and the shifted polar coordinate $(\rho, \varphi)$ (see Figure \ref{fig:schematic}), we define 
\begin{equation}\label{eq:def_u_0_components}
\begin{aligned}
    {u}_{0, \theta} (   z ,   r  ) &= \ep^2 \mu^{1 -s } \nu^{ \frac{1}{2}   } g (  \mu \rho )\sin(\varphi )     \\
    {u}_{0, r} (   z ,   r  ) &=- \ep^2 \mu^{1 -s } \nu^{ \frac{1}{2}   } f'  (  \mu  \rho  )  \p_z \rho    \\
    {u}_{0, z} (   z ,   r  ) &=   \ep^2 \mu^{1 -s } \nu^{ \frac{1}{2}   }   f '   (  \mu \rho  )   \p_r \rho 
\end{aligned}
\end{equation} 
where the profiles $f,g \in C^\infty_c(\RR )$ are fixed throughout the construction such that
\begin{equation}\label{eq:def_profiles}
\begin{cases}
  f '  (   \rho  )   =   1 \quad \text{for $ 1\leq \rho  \leq \frac{3}{2}$} & \\
  \Supp f   \subset \{  \frac{1}{2}\leq \rho  \leq 2 \}    & \\
 \Supp g \subset \{ 1\leq \rho  \leq \frac{3}{2} \} .
\end{cases}
\end{equation}

The key point is that the vector field $ \RR^2 \ni ( {u}_{0, r} , {u}_{0, z}        ) = C_{\ep,\mu} \nabla^{\perp} (f (\mu \rho) )  $ is a stationary solution of the 2D Euler equation on the $rz$-plane.

Precisely,  there exists a  pressure $p_0(r,z)$, smooth in $(r,z)$ and constant outside the support of $( {u}_{0, r} , {u}_{0, z}        )$, such that
\begin{equation}\label{eq:stationary_2Deuler}
\begin{cases}   
\big(  {u}_{ 0, r} \p_r    +  {u}_{0,  z} \p_z  \big){u}_{ 0, r } + \p_r {p}_0  = 0 & \\
 \big(  {u}_{ 0, r} \p_r    +  {u}_{0,  z} \p_z  \big)  {u}_{0,  z}  +  \p_z  {p}_0     = 0&\\
\p_r  {u}_{0, r} +  \p_z {u}_{0, z }= 0 &
\end{cases}
\quad\text{in $(r,z) \in \RR^2$} .
\end{equation}

We then define the initial data $u_0 :\RR^3 \to \RR^3 $ in cylindrical coordinates using \eqref{eq:def_u_0_components}:
\begin{equation}\label{eq:def_u_0}
\begin{aligned}
 {u}_0 & = {u}_{0, \theta } \et  +  {u}_{0, r} \er +  {u}_{0, z} \ez +  u_c \ez 
\end{aligned}
\end{equation}
where $u_c \ez  $ is a divergence corrector defined by 
\begin{equation}\label{eq:def_u_c}
u_c : =    \ep^2 \mu^{  -s } \nu^{ \frac{1}{2}   }  \frac{f     (  \mu \rho  ) }{r}      .
\end{equation}
We check that $u_0:\RR^3 \to \RR^3$ indeed has zero divergence:
\begin{equation}\label{eq:u_0_zero_div}
\begin{aligned} 
\D u_0 & = \D \big(  {u}_{0, r} \er +  {u}_{0, z} \ez +  u_c \ez \big) \\
&=  \p_r  {u}_{0, r}   + \frac{ {u}_{0, r} }{r} + \p_z {u}_{0, z}  +   \p_z u_c     \\
&=   \frac{ {u}_{0, r} }{r}  +   \p_z u_c  =  \ep^2 \mu^{1 -s } \nu^{ \frac{1}{2}   } \Big(  -\frac{ f'  (  \mu  \rho  )  \p_z \rho  }{r}      +  \frac{ f'  (  \mu  \rho  )  \p_z \rho  }{r}   \Big) \\
& = 0 .
\end{aligned}
\end{equation}

\begin{remark} 
  The initial data  ${u}_{0} $ is axisymmetric in $\RR^3$, essentially consisting of a vortex ring and a rotational flow,  supported inside a small torus with distance to origin $\sim \nu^{-1}$ and its cross-section having a diameter $ \sim \mu^{-1} $.

\end{remark}

\subsection{Basic estimates}
Observe that   \eqref{eq:def_u_0_components} and the fact that $ \nu \ll \mu$ imply the objects ${u}_{0, \theta } , {u}_{0, r } , {u}_{0, z }, u_c  $ are supported near a ring of radius $r \sim \nu^{-1}$ in $\RR^3$. More precisely,
\begin{equation}\label{eq:u_0_support}
\begin{cases}
\Supp {u}_{0 } \subset \mathcal{R}_{\mu}\subset \RR^3  & \\
\mathcal{R}_{\mu} : = \{ x \in \RR^3 : \dist(x, \mathcal{C}_{\nu} ) \leq 4 \mu^{-1}  \} & \\
\mathcal{C}_{\nu} : = \{x \in \RR^3 : r =\nu^{-1}\; \text{and}\; z=0  \}.
\end{cases}
\end{equation}

Using \eqref{eq:diff_rho_varphi} and the bounds $| \nabla^k ( r^\alpha ) | \lesssim r^{\alpha -k}   $, we have   the following estimates in $\RR^3$
\begin{equation}\label{eq:est_u_0_components_theta_r}
\begin{aligned}
  \nu^{\alpha} | r^\alpha f |_{W^{k,p}(\RR^3)} & \lesssim  \ep^2 \mu^{ k -s }    \Big( \mu^{1 - \frac{2}{p} } \nu^{\frac{1}{2} - \frac{1}{p} }  \Big)  \quad \text{for}\quad  f\in \{ {u}_{0, \theta } , {u}_{0, r }   \}  
\end{aligned}
\end{equation} 
and
\begin{equation}\label{eq:est_u_0_components_z}
\begin{aligned}
 |  {u}_{0, z } |_{W^{k,p}(\RR^3)} & \lesssim  \ep^2 \mu^{ k -s }   \Big( \mu^{1 - \frac{2}{p} } \nu^{\frac{1}{2} - \frac{1}{p} }  \Big)  \\
 |   {u}_{c}  |_{W^{k,p}(\RR^3)} & \lesssim  \ep^2 \mu^{ k -s }  (\mu^{-1} \nu)  \Big( \mu^{1 - \frac{2}{p} } \nu^{\frac{1}{2} - \frac{1}{p} }  \Big)  
\end{aligned}
\end{equation} 
where  the parameters $  k \in \NN $, $1 \leq p \leq \infty$, $\alpha \in \RR$ and  the constants do not depend on $\ep,\mu $ or $\nu$.

We emphasize again that all Sobolev or Lebesgue norms in this paper are taken in $\RR^3$ and hence we will not spell them out.

\begin{lemma}\label{lemma:estimates_u_0}
The vector field ${u}_0:\RR^3 \to \RR^3$ defined by  \eqref{eq:def_u_0} is smooth, divergence-free, and compactly supported around the origin with support of size $\sim \mu^{-2} \nu^{-1}$.

In addition, for any $k\geq 0$ and $1\leq p\leq \infty$, ${u}_0$ satisfies   the estimates
\begin{equation}\label{eq:lemma_u_0_1}
 |{u}_0 |_{W^{k,p}  } \leq C_{k,p} \ep^2 \mu^{ k -s  }  \Big( \mu^{1 - \frac{2}{p} } \nu^{\frac{1}{2} - \frac{1}{p} }  \Big) ,
\end{equation}
where the constant $C_{k,p} > 0 $ is independent of $\ep, \mu$ and $\nu$.

In particular, there exists a universal constant $\ep_0   \in (0,1)$ such that for any $0<\ep \leq \ep_0$, 
\begin{equation}\label{eq:lemma_u_0_2}
 |{u}_0|_{H^s  } \leq \ep  .
\end{equation}

\end{lemma}
\begin{proof}
The zero divergence of $u_0$ was proved in \eqref{eq:u_0_zero_div}. The size of the support follows from \eqref{eq:u_0_support}.

By reducing to the case $k \in \NN$, we need to estimate
\begin{equation}\label{eq:xux_lemma_u_0_1}
 |{u}_0 |_{W^{k,p}  }  \leq  |{u}_{0, \theta } \et |_{W^{k,p}  } +  | {u}_{0, r } \er |_{W^{k,p}  } +  |  {u}_{0, z } \ez  |_{W^{k,p}  } +  |  {u}_{c } \ez  |_{W^{k,p}  }.
\end{equation}

Since the estimates of $u_{0,z}$ and $u_c$ were proved in \eqref{eq:est_u_0_components_z}, we focus on $ {u}_{0, r }\er$ and $ {u}_{0, \theta } \et$.

By the point-wise bounds $ | \nabla^k \er | \lesssim r^{-k}  $ and $ | \nabla^k \et | \lesssim r^{-k}  $ for $k\in \NN$, the estimates of $ {u}_{0, r }\er$ and $ {u}_{0, \theta } \et$  follow  from \eqref{eq:est_u_0_components_theta_r}. We conclude that \eqref{eq:lemma_u_0_1} holds.

To show \eqref{eq:lemma_u_0_2}, we just set $k=s$ and $p=2$ in \eqref{eq:lemma_u_0_1}, then take $\ep>0$ small to absorb the universal constant.

\end{proof}

\subsection{The approximate solution}

Due to the anisotropy in the initial  data ${u}_0$ and our choice of frequency parameters \eqref{eq:def_nu_lambda},   we can approximate the exact solution of \eqref{eq:euler} with the initial data $u_0   $ by the first order transport dynamics of the swirl component.

Define the   approximate solution $\overline{u} :\RR^+ \times \RR^3 \to \RR^3 $ by
\begin{equation}\label{eq:def_overline_u}
\begin{aligned}
\overline{u}(t,x) &=   \overline{u}_\theta  \et + \overline{u}_r \er +  \overline{u}_z \ez \\
&=  \overline{u}_\theta \e_\theta +  {u}_{0, r} \er   +  {u}_{0, z} \ez + u_c \ez   .
\end{aligned}
\end{equation}
where $ \overline{u}_\theta:\RR^+ \times \RR^3 \to \RR  $ is the unique smooth solution of the free transport equation on $( r,z) \in \RR^2 $
\begin{equation}\label{eq:def_overline_u_theta}
\begin{cases}
\p_t \overline{u}_\theta +   ( {u}_{0, z} \p_z  +  {u}_{0, r} \p_r   )   \overline{u}_\theta = 0 &\\
\overline{u}_\theta  |_{t = 0} =   {u}_{0,\theta} 
\end{cases}
 \text{for $(t, r,z) \in \RR^+ \times \RR^2 $}
\end{equation}
In other words, using $\rho \varphi$ variables by \eqref{eq:def_u_0_components},
\begin{equation}\label{eq:def_overline_u_theta_1}
 \overline{u}_\theta (t, r,z) = \ep^2 \mu^{1 -s } \nu^{ \frac{1}{2}   } g (  \mu \rho )\sin(\varphi  - t \rho^{-1} \ep^2 \mu^{1 -s } \nu^{ \frac{1}{2}   } )   .
\end{equation} 
 
\begin{remark}
Three remarks concerning $\overline{u}$:
\begin{itemize}

\item Only the swirl component $  \overline{u}_\theta $ evolves in time, the vortex ring part $  \overline{u}_r \er + \overline{u}_z \ez $ is stationary.

\item $\overline{u}$ is globally defined on the time axis and remains smooth at all times. And obviously $ \overline{u} |_{t = 0} = u_0 $.

\item All $L^p$ norms of $ \overline{u}$ are almost constant in time (constant in $L^p(dr dz d\theta) $ with the metric $dr dz d\theta$), but any norm with derivatives will grow as  $t >0$. This is the driving mechanism of our $\dot{H}^s$-norm inflation.

\end{itemize}

\end{remark}

Define the critical time $t^* > 0$:
\begin{equation}\label{eq:def_critical_t*}
 t^* =  \ep^{- N  -2 }  \mu^{ - 2  + s } \nu^{-\frac{1}{2}} .
\end{equation}
Here $ N >0 $ is a large exponent that we fix to be
\begin{equation}\label{eq:def_large_N}
 N = \max\{\frac{10}{s},100 \} .
\end{equation}
Note that the crucial property that $  t^* \to 0 $ as $\mu \to \infty$ due to \eqref{eq:def_nu_lambda}. As said in the introduction, we see that $ t^* \gg_\ep |\nabla u_0|_{L^\infty} $ from \eqref{eq:lemma_u_0_1}.

In the following,  only the constants with an   $ \ep $ subscript depend on $\ep$. All constants are independent of $\mu,\nu     $ but may change from line to line.

\begin{proposition}[Approximation]\label{prop:approximate}
The vector field $\overline{u}: [0,\infty) \times \RR^3 \to \RR^3$ is smooth, divergence-free, axisymmetric, and satisfies for any $k\geq 0$ and $1\leq p\leq \infty$, the estimates 
\begin{align}\label{eq:prop_error_1}
| \overline{u} (t) |_{W^{k,p}   }  \leq     C_{ k,p }  \ep^{2 - k N   }  \mu^{ k -s  } \mu^{1  - \frac{ 2}{p} } \nu^{\frac{1}{2} - \frac{1}{p} }    \quad \text{on the interval $t\in [0,t^*]$  }.
\end{align}

Moreover, $\overline{u}$ is an approximate solution of \eqref{eq:euler} in the following sense. There exist  smooth pressure $\overline{p}:   \RR^3  \to  \RR      $  and error field $\overline{E}: \RR^+ \times  \RR^3  \to  \RR^3    $, compactly supported in $\RR^3$ at each time $t \in \RR^+$, such that 
\begin{equation}\label{eq:prop_error_2}
\begin{cases}
   \p_t  \overline{u} + \overline{u}\cdot \nabla \overline{u}   + \nabla \overline{p} = \overline{E} & \text{in $(t,x) \in \RR^+ \times  \RR^3  $} \\
\D \overline{u}  = 0 &\\
\overline{u} |_{t = 0 } = u_0  & 
\end{cases}
\end{equation}
where  error field $\overline{E}$ satisfy   uniform in time estimates 
\begin{equation}\label{eq:prop_error_3}
    | \nabla^k \overline{E}(t)|_{L^2 } \leq   C_{\ep,k }   \mu^{k - s }  (\mu^{-1} \nu)   (\mu^{2-  s}    \nu^\frac{1}{2}  )  \quad \text{when $ 0 \leq t \leq t^*$} .
\end{equation}
 
\end{proposition}
\begin{proof}

We first prove \eqref{eq:prop_error_1} where we keep track of $\ep$-related constants and in the second part we prove \eqref{eq:prop_error_3} without tracking the dependence on $\ep$.

\noindent
\textbf{\underline{Part 1: Estimates of $\overline{u}$}}

Recall that $ \overline{u} =  \overline{u}_{  \theta } \et +  {u}_{  0, r } \er+  {u}_{  0, z } \ez + u_c \ez  $. By interpolation, it suffices to estimate \eqref{eq:prop_error_1} for $k \in \NN$. From  \eqref{eq:def_overline_u_theta_1} we also have  $\Supp \overline{u}  \subset   \mathcal{R}_{\mu}  
$ with the ring $ \mathcal{R}_{\mu}   \subset \RR^3$ from \eqref{eq:u_0_support}  having a measure $\sim \mu^{-2} \nu ^{-1}$.
Thus we can only consider the case $p = \infty$ and $k \in \NN$.

By the considerations in Lemma \ref{lemma:estimates_u_0}, the stationary part $  {u}_{ 0,   r } \er+  {u}_{ 0,  z } \ez + u_c \ez $ satisfies the bound \eqref{eq:prop_error_1} automatically. Moreover, by $|\nabla^k \et| \lesssim r^{-k}$, it suffices to bound $| \overline{u}_{   \theta } |_{W^{k,p}}$,
\begin{equation}\label{eq:aux_lemma_error_1}
\overline{u}_{   \theta }   =   \ep^2 \mu^{1 -s } \nu^{ \frac{1}{2}   } g (  \mu \rho )\sin(\varphi  - t \rho^{-1} \ep^2 \mu^{1 -s } \nu^{ \frac{1}{2}   } ) . 
\end{equation}

 When $k=0$, the estimates follow trivially.

For integers $k\geq 1$, observe from \eqref{eq:diff_rho_varphi} that a differentiation each time gives out a factor of 
\begin{equation}\label{eq:aux_lemma_error_2}
\max \{ \mu,   t \rho^{-2} \ep^2 \mu^{1 -s } \nu^{ \frac{1}{2}   } \} \leq  \max \{ \mu ,  C \ep^{2} t      \mu^{3 - s} \nu^{ \frac{1}{2}   } \}.
\end{equation} 

Since $t\leq t^*$, recalling the definition \eqref{eq:def_critical_t*}, we have that
\begin{equation}\label{eq:aux_lemma_error_3}
  t      \mu^{3 - s} \nu^{ \frac{1}{2}   }   \leq    \ep^{- N -2  }  \mu ,
\end{equation}
and we obtain the maximum factor for each differentiation is $ \ep^{- N   }  \mu $:
\begin{equation}\label{eq:aux_lemma_error_4}
| \nabla^k(  \overline{u}_{   \theta }  )| \leq   C_{  k } \ep^{2 - k N   }  \mu^{ k -s  }  \mu^{1 -s }     \nu^{ \frac{1}{2}   }
\end{equation}
where the constant is not dependent on $\ep$. From this estimate \eqref{eq:aux_lemma_error_4} we can recover all $W^{k,p}$ estimates for \eqref{eq:aux_lemma_error_1}, and hence \eqref{eq:prop_error_1} is proved.

\noindent
\textbf{\underline{Part 2: Estimates of $\overline{p}$ and $\overline{E}$}}

To find $\overline{p}$ and $\overline{E} $,  we rewrite  \eqref{eq:prop_error_2} in cylindrical coordinates, under the axial symmetry,
\begin{equation}\label{eq:aux_lemma_error_5}
\begin{cases}
\p_t \overline{u}_{  \theta } + \overline{u}_{  r} \p_r \overline{u}_{  \theta } + \overline{u}_{  z} \p_z \overline{u}_{  \theta }+ \frac{1}{r} \overline{u}_{  \theta }\overline{u}_{  r}  = \overline{E}_{  \theta } & \\
\p_t \overline{u}_{  r } + \overline{u}_{  r} \p_r \overline{u}_{  r } + \overline{u}_{  z} \p_z \overline{u}_{  r } - \frac{1}{r} \overline{u}_{  \theta }^2 + \p_r \overline{p} = \overline{E}_{  r} & \\
\p_t \overline{u}_{  z} + \overline{u}_{  r} \p_r \overline{u}_{ z} + \overline{u}_{  z} \p_z \overline{u}_{  z}  +  \p_z   \overline{p}  = \overline{E}_{ z} .&  
\end{cases}
\end{equation}
By the definitions  $ \overline{u}_r = {u}_{0,r}$  and $ \overline{u}_z = {u}_{0,z} + u_c$ are stationary, the above system is equivalent to 
\begin{equation}\label{eq:aux_lemma_error_5a}
\begin{cases}
\p_t \overline{u}_{  \theta } +  \big(  {u}_{ 0, r} \p_r    +  {u}_{0,  z} \p_z  \big) \overline{u}_{  \theta }+ \big(  {u}_{ c} \p_z \overline{u}_{  \theta }+ \frac{1}{r} \overline{u}_{  \theta }{u}_{0,r} \big)  = \overline{E}_{  \theta } & \\
\big(  {u}_{ 0, r} \p_r    +  {u}_{0,  z} \p_z  \big){u}_{ 0, r } + \p_r \overline{p} +\big(   {u}_{c} \p_z  {u}_{ 0, r }- \frac{1}{r} \overline{u}_{  \theta }^2 \big) = \overline{E}_{  r} & \\
 \big(  {u}_{ 0, r} \p_r    +  {u}_{0,  z} \p_z  \big)  {u}_{0,  z}  +  \p_z   \overline{p}  &\\
  \qquad +   \big(  {u}_{ 0, r} \p_r    +  {u}_{0,  z} \p_z +u_c\p_z \big) u_c +   {u}_{ c} \p_z  {u}_{0,z }   = \overline{E}_{ z}   .&
\end{cases}
\end{equation}

By \eqref{eq:stationary_2Deuler}, the vector field $( {u}_{  0,r},{u}_{  0,z }  )  \in \RR^2$ is a stationary solution of the 2D Euler in $rz$-plane with the pressure $p_0$. To equate both sides of \eqref{eq:aux_lemma_error_5a}, let us  define $\overline{E} = \overline{E}_\theta \et  +  \overline{E}_r \er  + \overline{E}_z \ez $ by  
\begin{align}\label{eq:aux_lemma_error_6}
\overline{E}_{ \theta }  & = {u}_{ c} \p_z \overline{u}_{  \theta } +  \frac{1}{r} \overline{u}_{  \theta } {u}_{0,  r}   \\
 \overline{E}_{ r} & =  {u}_{c} \p_z  {u}_{ 0, r }- \frac{1}{r} \overline{u}_{  \theta }^2  \\
\overline{E}_{ z} &= {u}_{  0,r} \p_r u_c + {u}_{ 0, z} \p_z  {u}_{c}+ {u}_{ c} \p_z  {u}_{c}    + {u}_{ c} \p_z  {u}_{0,z }
\end{align}
 and $ \overline{p}(r,z) := p_0(r,z)$  with $p_0(r,z)$ from \eqref{eq:stationary_2Deuler}. Note that $\overline{p}$ is also smooth as a function $\RR^3 \to \RR$ since $p_0(r,z)$ is constant near $r=0$.

It remains to show the estimate for $\overline{E}   $. By the considerations in \eqref{eq:aux_lemma_error_2} and \eqref{eq:aux_lemma_error_3}, it suffices to only consider the case $k =0$ for $ | \overline{E} |_{L^2 } $.

Since   estimating  $\overline{E}_\theta$ and $\overline{E}_z$ is very  similar, we only demonstrate the estimate of $\overline{E}_r $.

 For $ \overline{E}_r$, we   use the H\"older inequality and   the fact that $r \sim \nu^{-1}$ on the support of $\overline{u}_{    \theta } $, 
\begin{equation}\label{eq:aux_lemma_error_10}
\begin{aligned}
 | \overline{E}_r |_{L^2 }&  \leq |  {u}_{c}|_{L^2 } | \p_z  {u}_{ 0, r }  |_{L^\infty  } +    | \frac{1}{r}  \overline{u}_{    \theta }   |_{L^2 }   |\overline{u}_{   \theta }   |_{L^\infty  }\\
&  \lesssim |  {u}_{c}|_{L^2 } |\nabla  {u}_{ 0, r }  |_{L^\infty  }+ \nu   |   \overline{u}_{    \theta }   |_{L^2 }   |\overline{u}_{   \theta }   |_{L^\infty  }.
\end{aligned}
\end{equation}
By  \eqref{eq:est_u_0_components_theta_r},\eqref{eq:est_u_0_components_z}, and  \eqref{eq:prop_error_1} proved in Part 1,
\begin{equation}\label{eq:aux_lemma_error_11}
\begin{aligned}
| \overline{E}_r |_{L^2 }&      \lesssim \big( \mu^{-1}\nu   \mu^{-s}  \mu^{1-s} \mu  \nu^{ \frac{1 }{2}}  \big) + \nu \big( \mu^{-s}  \mu^{  -s} \mu  \nu^{ \frac{1 }{2}}  \big)  \\
&\lesssim  \mu^{1- 2 s   } \nu^\frac{3}{2}  .
\end{aligned}
\end{equation}

\end{proof}

In the below lemma, we show that the approximate solution $ \overline{u} $ develops the desired norm inflation at the critical time $t^*$.
\begin{lemma}\label{lemma:Hs_lowerbound}
There exists a universal constant $\ep_0   \in (0,1)$ such that for any $0<\ep \leq \ep_0$,   there holds the lower bound
\begin{align}\label{eq:lemma_Hs_lowerbound_1}
|  \overline{u} (t^*)|_{\dot{H}^{s}} \geq  \ep^{-2}
\end{align}
where $t^* > 0$ is the critical time defined in \eqref{eq:def_critical_t*}.

\end{lemma}
\begin{proof}

To prove  \eqref{eq:lemma_Hs_lowerbound_1}, we first prove a lower bound for  $ | \overline{u} |_{\dot{H}^1 } $.

By the definitions of $ \dot{H}^1$ norm and of $\overline{u}$, we have
$$
 | \overline{u} |_{\dot{H}^1 } =  | \nabla  \overline{u} |_{L^2}   \geq  | \nabla (\overline{u}_\theta \et) |_{L^2} -   | \nabla (  {u}_{0 , r } \er + {u}_{0 , z } \ez+  {u}_{c } \ez) |_{L^2 } 
$$
and hence 
\begin{equation}\label{eq:aux_lemma_Hs_lowerbound_1}
\begin{aligned} 
 | \overline{u} |_{\dot{H}^1 } &  \geq | \nabla \overline{u}_\theta  |_{L^2 }  - C |  r^{-1} \overline{u}_\theta  |_{L^2 } -  C |  r^{-1}  {u}_{0 , r }  |_{L^2 }    -| \nabla {u}_{c }  |_{L^2 } \\
&\qquad \qquad -   | \nabla    {u}_{0 , r } |_{L^2 }     -  | \nabla {u}_{0 , z }  |_{L^2 }  .
\end{aligned}
\end{equation}
By   \eqref{eq:est_u_0_components_theta_r} and \eqref{eq:est_u_0_components_z}, the   terms on the last line of \eqref{eq:aux_lemma_Hs_lowerbound_1} are bounded from below by $ -  C \ep^2  \mu^{1 -s }$. Similarly, the negative three terms on the first line are bounded from below by $-C  \ep^2  \mu^{  - s} \nu$, which is even better than $- C \ep^2  \mu^{1 - s}$. Therefore, it suffices to show
\begin{equation}\label{eq:aux_lemma_Hs_lowerbound_1a}
| \nabla   \overline{u}_\theta  (t^*) |_{L^2 } \geq C  \ep^{2 -   N   }  \mu^{ 1 -s  } 
\end{equation}
which dominates $\ep^2  \mu^{1 - s }$ for $\ep>0$ small. 

By the point-wise orthogonality $| \nabla f|^2 = |(\p_r f, \p_z f)|^2 =  | \p_\rho f |^2 + |\frac{1}{\rho}  \p_\varphi f |^2 $ for any axisymmetric scalar function $f:\RR^3 \to \RR$, we can focus on such a lower bound for $ \p_\rho  \overline{u}_\theta $. A direct computation shows,  
\begin{align*}
\p_\rho  \overline{u}_{\theta }  =  & \ep^2 \mu^{1 -s } \nu^{ \frac{1}{2}   } g (  \mu \rho )\cos(\varphi  - t \rho^{-1} \ep^2 \mu^{1 -s } \nu^{ \frac{1}{2}   } ) t \rho^{-2} \ep^2 \mu^{1 -s } \nu^{ \frac{1}{2}   }   \\
&  +     \ep^2 \mu^{1 -s } \nu^{ \frac{1}{2}   } \mu  g' (  \mu \rho )\sin(\varphi  - t \rho^{-1} \ep^2 \mu^{1 -s } \nu^{ \frac{1}{2}   } ),  
\end{align*} 
and integrating in the $\theta, \rho ,\varphi $ variables and treating the second term above as an error, we have
\begin{align}\label{eq:aux_lemma_Hs_lowerbound_2}
| \p_\rho  \overline{u}_{\theta } (t^*)|_{L^2} \geq C \ep^2\mu^{  -s }   \Big(   t^*  \ep^2 \mu^{3 -s } \nu^{ \frac{1}{2}   }  -    c \mu    \Big)  .
\end{align} 
By the definition of $t^*$ from \eqref{eq:def_critical_t*}, $t^*  \ep^2 \mu^{3 -s } \nu^{ \frac{1}{2}   } \geq  \ep^{- N} \mu^{ 1}$.  It is evident that for all    $\ep>0$ sufficiently small,  the first term above dominates the second one, and  we have
\begin{equation}\label{eq:aux_lemma_Hs_lowerbound_3}
| \p_\rho  \overline{u}_{\theta } (t^*)|_{L^2} \geq C  \ep^{2 -   N   }  \mu^{ 1 -s  }    .
\end{equation}

From \eqref{eq:aux_lemma_Hs_lowerbound_1}, \eqref{eq:aux_lemma_Hs_lowerbound_1a} and \eqref{eq:aux_lemma_Hs_lowerbound_3}, for any    $\ep>0$ sufficiently small 
\begin{equation}\label{eq:aux_lemma_Hs_lowerbound_4}
 | \overline{u} (t^*) |_{\dot{H}^1 } \geq  C  \ep^{2 -   N   }  \mu^{ 1 -s  }   .
\end{equation}

We now use the lower bound \eqref{eq:aux_lemma_Hs_lowerbound_4} to deduce the sought estimate \eqref{eq:lemma_Hs_lowerbound_1}. We split the argument into two cases.

\noindent
\textbf{\underline{Case 1: $s \geq 1$}}

By the Sobolev interpolation, 
\begin{equation}\label{eq:aux_lemma_Hs_lowerbound_11}
 | \overline{u} |_{\dot{H}^1} \leq  C | \overline{u} |_{L^2 }^{1 - \frac{1}{s}}  | \overline{u} |_{ \dot{H}^s }^{ \frac{1}{s}} \quad \text{when $s \geq 1$} .
\end{equation}

Then since $1-s \leq 0$,    \eqref{eq:prop_error_1} for $L^2$ and \eqref{eq:aux_lemma_Hs_lowerbound_4} imply
\begin{equation}\label{eq:aux_lemma_Hs_lowerbound_13}
\begin{aligned}
  | \overline{u} (t^*) |_{\dot{H}^s } & \geq | \overline{u} |_{H^1}^s   | \overline{u} |_{L^2 }^{ 1 - s}  \\
& \geq \big(  \ep^{2-N } \mu^{1-s} \big)^s \big(  \ep^{2  } \mu^{ -s} \big)^{1-s} \\
& \geq       C \ep^{ 2-N s} .
\end{aligned}
\end{equation}

\noindent
\textbf{\underline{Case 2: $0< s \leq 1$}}

 In this case, we use the interpolation
\begin{equation*}
 | \overline{u} |_{\dot{H}^1} \leq  C | \overline{u} |_{\dot{H}^s }^{\frac{1}{2-s}} | \overline{u} |_{\dot{H}^2}^{\frac{ 1 - s }{2-s}}   ,
\end{equation*}
so that
\begin{equation}\label{eq:aux_lemma_Hs_lowerbound_22}
 | \overline{u} |_{\dot{H}^s} \geq  C | \overline{u} |_{\dot{H}^1 }^{ {2-s} } | \overline{u} |_{\dot{H}^2}^{ {   s-1  } }   .
\end{equation}

Since $ s  - 1  \leq 0$, from \eqref{eq:aux_lemma_Hs_lowerbound_4}   and  \eqref{eq:prop_error_1} again, it follows that
\begin{equation}\label{eq:aux_lemma_Hs_lowerbound_24}
\begin{aligned}
  | \overline{u} (t^*) |_{\dot{H}^s }  & \geq   \big(  \ep^{2-N } \mu^{1-s} \big)^{2-s} \big(  \ep^{2 -2N } \mu^{ 2-s} \big)^{s - 1 } \\
  & \geq    C \ep^{ 2-N s}  .
\end{aligned}
\end{equation}

In either case, \eqref{eq:lemma_Hs_lowerbound_1} follows from the fact $Ns \geq 10 $.

\end{proof}

\section{Perturbation analysis}\label{sec:no_blowup_proof}

In this section, we prove that the approximate solution $ \overline{ u} $ defined in the previous section stays close to the actual solution $u$ of \eqref{eq:euler}, at least until the critical time $t^*$,   the onset of $\dot H^s$-norm inflation.

\subsection{The main proposition}

We will be proving the following proposition, the main result of this section.
\begin{proposition}\label{prop:no_blowup}

Let $T=T(u_0 )> 0 $ be the maximal time of existence for the local-in-time solution $u: [0 , T ) \times \RR^3 \to \RR^3$ of \eqref{eq:euler} with the initial data $u_0$ defined by \eqref{eq:def_u_0}.

For any $\ep>0$, there exists $\mu_0>0$ sufficiently large such that if $\mu \geq \mu_0$, then there must be $ 0<   t^*< T$, namely $ u  \in C^\infty( [0,t^* ] \times \RR^3 )$.

More quantitatively, for any $\ep>0$, if $\mu \geq \mu_0$, then 
\begin{equation}\label{eq:prop_no_blowup}
|u - \overline{u}|_{L^\infty ([0,t^*]; H^k) } \leq C_{k,\ep } \mu^{k - s} 
\end{equation}
for any $k \geq 1$ where $C_{k,\ep }$ is independent of $\mu$.
\end{proposition}

It is straightforward to obtain Theorem \ref{thm:Hs} once we combine Proposition \ref{prop:no_blowup} with Proposition \ref{prop:approximate} and Lemma \ref{lemma:estimates_u_0}, by taking $\ep>0$ possibly smaller and then $\mu\geq 1$ large. 

The proof of Proposition \ref{prop:no_blowup} will be broken into several small lemmas below. The essence is a bootstrap argument that ``propagates'' the smallness at lower $L^2$ norm of the difference $u - \overline{u} $ to all higher Sobolev $H^k$ norm on the interval $[0,t^*]$.

\subsection{The bootstrap assumption}

We now introduce the bootstrap assumption.

Denote by $w  = u - \overline{u}$ the difference between the approximate solution $ \overline{u}$ and the exact solution $u$ (having the same initial data $u_0$ as $ \overline{u} $) . It follows from Proposition \ref{prop:approximate} that the vector field $w$, well-defined on the time interval $[0,T)$,  satisfies the evolution equation
\begin{equation}\label{eq:aux_noblowup_0}
\begin{cases}
\p_t w + u\cdot \nabla w + w \cdot \nabla  \overline{u} + \nabla q = \overline{E} & \\
\D w  = 0 & \\
w |_{t = 0 } = 0 &
\end{cases}
\quad \text{for $(t,x) \in [0,T ) \times \RR^3 $}
\end{equation}
where $q : = p - \overline{p}$ with $p$ being the pressure of the exact solution $u$ and $\overline{p}$ the pressure from Proposition \ref{prop:approximate}.

In view of \eqref{eq:prop_error_1} from Proposition \ref{prop:approximate}, for any given $\ep>0$ let us fix $M_\ep \geq 1$  such that 
\begin{equation}\label{eq:aux_noblowup_1}
|\nabla \overline{u} (t)|_{L^\infty} \leq M_\ep \mu^{2 -s} \nu^\frac{1}{2}  \quad \text{for all} \quad    0\leq t    \leq t^*  .
\end{equation}

We will prove the following bound:
\begin{equation}\label{eq:aux_noblowup_claim_0}
  |\nabla u (t) |_{L^\infty } \leq 2   M_\ep  \mu^{2 -s} \nu^\frac{1}{2}   \quad \text{for all} \quad    0\leq t    \leq t^*  
\end{equation}  
by a continuity argument. Since   $|\nabla  {u}_0  |_{L^\infty} \leq M_\ep \mu^{2 -s} \nu^\frac{1}{2} $ holds at $t = 0$, let us introduce

\mdfsetup{skipabove=2pt,skipbelow=2pt}
\begin{mdframed}[linewidth=1pt,frametitle={The bootstrap assumption:},nobreak=true]
 
For some $ 0< t_0 \leq t^*$, there holds 
\begin{equation}\tag{$\dagger$}\label{eq:aux_noblowup_claim}
  |\nabla u (t) |_{L^\infty } \leq 2   M_\ep  \mu^{\frac{3}{2} -s} \nu   \quad \text{for all} \quad    0\leq t    \leq t_0  .
\end{equation}

\end{mdframed}

In the steps below, we will prove that if $\mu $ is sufficiently large (depending on $\ep>0$), then under the bootstrap assumption \eqref{eq:aux_noblowup_claim},  we  have the improved bound
$$
  |\nabla u (t) |_{L^\infty } \leq    \frac{3}{2} M_\ep  \mu^{2 -s} \nu^\frac{1}{2}   \quad \text{for all} \quad    0\leq t    \leq t_0  ,
$$
which will imply   \eqref{eq:aux_noblowup_claim} on the whole interval $[0, t^*]$ by continuity.

\subsection{Basic estimates}

To facilitate the estimates, we will frequently use the following bounds,  direct consequences of Proposition \ref{prop:kato_ponce}, Lemma \ref{lemma:estimates_u_0}, the definition of $t^*$ (from \eqref{eq:def_critical_t*}), and the bootstrap assumption \eqref{eq:aux_noblowup_claim},
\begin{equation}\label{eq:aux_noblowup_assumption_0}
\begin{cases}
| u (t)|_{ W^{k, p} }   & \leq C_{\ep, k,p }     \mu^{ k -s  } \mu^{ 1  - \frac{2}{p}  } \nu^{\frac{1}{2} - \frac{1}{p}  }  ,   \\
t \mu^{2  - s } \nu^\frac{1}{2}  & \leq C_\ep 
\end{cases}
\quad \text{for any}\,\,  0\leq t    \leq t_0 
\end{equation}
where $k \geq 0$  and $ 1<p<\infty $.

\subsection{$L^2$ estimates}
As the first step in the bootstrap argument, we show $w$ is small in $L^2$, gaining a small factor $  (\mu^{-1}   \nu  )   $   comparing to $ | u|_{L^2}$ and $ |  \overline{{u}}|_{L^2}$.

\begin{lemma}\label{lemma:bootstrap_L2}
Under the bootstrap assumption \eqref{eq:aux_noblowup_claim}, the difference $w = u - \overline{u}$ satisfies
\begin{equation}\label{eq:lemma_bootstrap_L2}
 |w(t) |_{L^{2}}  \leq C_\ep     \mu^{ - s } \big(\mu^{-1}  \nu  \big)   \quad \text{for any}\,\,  0\leq t    \leq t_0 .
\end{equation}

\end{lemma}
\begin{proof}

Multiplying \eqref{eq:aux_noblowup_0} by $ w = u - \overline{u}$ and integrating, we obtain 
\begin{equation}\label{eq:aux_lemma_bootstrap_L2_1}
\frac{d}{dt}|w(t) |_{L^{2}}^2 \lesssim   |\nabla \overline{u}  |_{L^\infty } | w |_{ L^{2} }^2     +  | E |_{L^{2}}| w |_{L^{2}}
\end{equation}
where we have used the fact that both $u$ and  $\overline{u} $ are divergence-free.  

By  Proposition \ref{prop:approximate},  we have
\begin{equation}\label{eq:aux_lemma_bootstrap_L2_2}
\begin{aligned}
\frac{d}{dt}|w(t) |_{L^{2}}  & \lesssim_\ep \mu^{2 -s} \nu^\frac{1}{2} | w |_{ L^{2} }      +\mu^{   -   s}   (\mu^{-1}   \nu )  \mu^{2  -   s} \nu^{\frac{1}{2}}.
\end{aligned} 
\end{equation}

By Gronwall's inequality (recalling $w |_{t=0} = 0 $), from \eqref{eq:aux_lemma_bootstrap_L2_2} we have that
\begin{equation}\label{eq:aux_noblowup_7}
\begin{aligned}
 |w(t) |_{L^{2}}   & \lesssim  e^{C_\ep t \mu^{2 -s} \nu^\frac{1}{2}  }   \Big(   t \mu^{   -   s}   (\mu^{-1}   \nu   )  \mu^{2  -   s} \nu^\frac{1}{2}   \Big)            \quad \text{for any $t \in [0, t_0 ]$}  .
\end{aligned}
\end{equation}

The sought estimate then follows from \eqref{eq:aux_noblowup_assumption_0},
\begin{equation*}
 |w(t) |_{L^{2}}   \leq C_\ep     \mu^{ -   s } \big(\mu^{-1}  \nu  \big)   \quad \text{for any $t \in [0, t_0 ]$} .
\end{equation*}

\end{proof}

\subsection{$H^k$ estimates}

In this step, we propagate the small factor $ (\mu^{-1}  \nu   )   $ in $  |w  |_{L^{2}} $ to higher Sobolev   norms of $w$. Due to the failure of \eqref{eq:aux_noblowup_assumption_0} at the end-point $p =\infty$, we have a weaker bound \eqref{eq:lemma_Hk_w} when $ k \geq 1$.

However, it is crucial that the power of the smallness factor $\mu^{-1} \nu$ is strictly bigger than $\frac{1}{2}$--- this is needed to compensate the lossy embedding $H^{\frac{5}{2} + } \hookrightarrow W^{1,\infty } $ in our anisotropic setup.

\begin{lemma}\label{lemma:bootstrap_Hk}
Under the bootstrap assumption \eqref{eq:aux_noblowup_claim}, the difference $w = u - \overline{u}$ satisfies for any integer $k \geq 0$ the estimate,
\begin{equation}\label{eq:lemma_Hk_w}
\begin{aligned}
 |\nabla^k w(t) |_{L^{2}}  \leq  C_{\ep, k }     \mu^{k - s } \big(\mu^{-1}  \nu  \big)^{\delta_k}  \quad \text{for any $t \in [0, t_0 ]$}
\end{aligned}
\end{equation}
where $\delta_k  = \frac{9 + 10^{- k  }}{10 }       > 0$.

\end{lemma}
\begin{proof}

We will prove \eqref{eq:lemma_Hk_w} by induction. 

Let $k \geq 1$ as the case $k=0$ was proved in Lemma \ref{lemma:bootstrap_L2}. Now we assume \eqref{eq:lemma_Hk_w} has been proved for levels $\leq k-1$ with $\mu,\nu$ independent constants.

Let $\p^\alpha$ be an order $k$ partial derivative. We differentiate \eqref{eq:aux_noblowup_0} by $\p^\alpha$ to obtain that
\begin{equation}\label{eq:aux_noblowup_91}
\begin{aligned}
 &   \p_t \p^\alpha w +  {u} \cdot \nabla \p^\alpha w \\
& =   \p^\alpha E -  \nabla \p^\alpha p   - \p^\alpha w \cdot \nabla   \overline{u} \\
& \qquad +    \sum_{  \beta    } C_{ \alpha,  \beta  } \p^ \beta  u  \cdot \nabla  \p^{\alpha -  \beta  } w    +   C'_{ \alpha,  \beta  }\p^{ \alpha -   \beta}  w  \cdot \nabla  \p^{  \beta  } \overline{u}  
\end{aligned}
\end{equation}
where the   summation  multi-index $\beta\in \NN^3$ runs over the  range $\beta_i \leq \alpha_i$ and $|\beta| >0 $ and $C_{ \alpha,  \beta  }, C'_{ \alpha,  \beta  }$ are constants from using the product rules.

We multiply \eqref{eq:aux_noblowup_91} by $\p^\alpha w $,  integrate in space, and use the H\"older inequality to obtain
\begin{equation}\label{eq:aux_noblowup_20}
\begin{aligned}
    \frac{d}{dt} |  \p^\alpha w  |_{L^{2}}^2 & \lesssim_\alpha          |\p^\alpha E|_{L^2}   |  \p^\alpha w  |_{L^{2}}  + | \nabla   \overline{u}  |_{L^{\infty }}   |  \p^\alpha w  |_{L^{2}}^2  \\
& \quad +    |  \p^\alpha w  |_{L^{2}} \sum_{1  \leq  m \leq k }   |  \nabla^m u  |_{L^\infty}      |  \nabla^{ k+1  - m } w  |_{L^{2}}     \\
&   \quad\quad     +  |  \p^\alpha w  |_{L^{2}} \sum_{1  \leq  m \leq k }   |   \nabla^{ m +1}     \overline{u}  |_{L^\infty}  |\nabla^{k - m }  w   |_{L^{2}}    .
\end{aligned}
\end{equation}
Note that in \eqref{eq:aux_noblowup_20} we have used that $ \langle {u} \cdot \nabla \p^\alpha w,  \p^\alpha w  \rangle = 0$ as $u$ is divergence-free, so the right-hand side only has  $w$ terms with up to $k$ derivatives.

Summing \eqref{eq:aux_noblowup_20} over all multi-indexes  $|\alpha| = k$ and simplifying, we have that
\begin{equation} 
\begin{aligned}
    \frac{d}{dt} |  \nabla^k  w  |_{L^{2}}^2  & \lesssim_k  I + J 
\end{aligned}
\end{equation}
where $I$  contains the linear amplification and the source term:
\begin{equation}\label{eq:aux_noblowup_20a}
\begin{aligned}
& I =    | \nabla   \overline{u}  |_{L^{\infty }}   |  \nabla^k w  |_{L^{2}}^2          + |   \nabla u |_{L^{\infty }}  |  \nabla^k w  |_{L^{2}}^2   + | \nabla^k E |_{L^{2}}|\nabla^k w |_{L^{2}} 
\end{aligned}
\end{equation}
and $J$ consists of input from lower order derivatives of $w$:
\begin{equation}\label{eq:aux_noblowup_20b}
\begin{aligned}
&   J =   |  \nabla^k w  |_{L^{2}} \sum_{ 0 \leq   m  \leq k -1 }    |\nabla^m  w   |_{L^{2}}  \big(   |   \nabla^{k +1 - m }     \overline{u}  |_{L^\infty}  +|   \nabla^{k+1 - m }      {u}  |_{L^\infty}    \big) .
\end{aligned}
\end{equation}
Note that the term $|  w   |_{L^{2}}     |   \nabla^{k +1 }      {u}  |_{L^\infty}   $ was absent in \eqref{eq:aux_noblowup_20} but is added in the summation of \eqref{eq:aux_noblowup_20b} for the  ease of notations.

\noindent
\textbf{\underline{Estimates of $I$}}

For the term $I$, using \eqref{eq:aux_noblowup_assumption_0} and Proposition \ref{prop:approximate} we have
\begin{equation}\label{eq:aux_noblowup_I_1}
 I \leq  C_{\ep, k }      \mu^{2 -s} \nu^\frac{1}{2}  |  \nabla^k w  |_{L^{2}}^2     + | \nabla^k E |_{L^{2}}|\nabla^k w |_{L^{2}}  .
\end{equation}
Then for the last term in \eqref{eq:aux_noblowup_I_1}, we use Young's inequality together with \eqref{eq:prop_error_3} from Proposition \ref{prop:approximate} to obtain 
\begin{equation}\label{eq:aux_noblowup_I_2}
\begin{aligned}
| \nabla^k E |_{L^{2}}|\nabla^k w |_{L^{2}}  & \lesssim   \mu^{2 -s} \nu^\frac{1}{2}    |\nabla^k w |_{L^{2}}^2  + (  \mu^{2 -s} \nu^\frac{1}{2}  )^{-1} | \nabla^k E |_{L^{2}}^2 \\
&  \lesssim   \mu^{2 -s} \nu^\frac{1}{2}   |\nabla^k w |_{L^{2}}^2  + \mu^{2k-2s  } (  \mu^{2 -s} \nu^\frac{1}{2}   )   ( \mu^{-1} \nu )^2    .
\end{aligned}
\end{equation}
From \eqref{eq:aux_noblowup_I_1} and \eqref{eq:aux_noblowup_I_2}, it follows that
\begin{equation}\label{eq:aux_noblowup_I_3}
 I \leq  C_{\ep, k }  \Big(   \mu^{2 -s} \nu^\frac{1}{2}    |  \nabla^k w  |_{L^{2}}^2     + \mu^{2k-2s  } (  \mu^{2 -s} \nu^\frac{1}{2}  )   ( \mu^{-1} \nu )^2     \Big)   .
\end{equation}

\noindent
\textbf{\underline{Estimates of $J$}}

For $J$, we first need to estimate  factor $   |   \nabla^{m }     \overline{u}  |_{L^\infty}  +|   \nabla^{ m }      {u}  |_{L^\infty}      $ for integer $0 \leq m \leq k +1 $.

Since we need $ L^\infty  $-type estimates on $u$, we need to give up some ($k$-dependent) small exponent to compensate for the failure of \eqref{eq:aux_noblowup_assumption_0} at $ p = \infty$.

By the Sobolev embedding $ W^{1+\frac{3}{p} +\delta, p} (\RR^3) \hookrightarrow W^{1,\infty }(\RR^3)$  for any $\delta >0$ and $p<\infty$, it follows from \eqref{eq:aux_noblowup_assumption_0} that for any $ m  \in \NN$,
\begin{equation}\label{eq:aux_noblowup_22b}
\begin{aligned}
|\nabla^m u (t)|_{ L^\infty } & \lesssim_{\delta, p } |  u (t)|_{ W^{m+\frac{3}{p} +\delta, p} }    \\
& \lesssim_{m,\delta,p}  \mu^{ m   -s +\delta   } \mu^{1 + \frac{1}{p}     } \nu^{\frac{1}{2} - \frac{1}{p} }  \quad \text{for all}\,\,  0\leq t    \leq t_0  .
\end{aligned}
\end{equation}

In \eqref{eq:aux_noblowup_22b}, we have the freedom to choose $p<\infty$ and $\delta>0$, so by the continuity of the exponent and that $ \nu =  \mu^{1-b }$, we can choose $p <\infty, \delta >0$ depending on $k$ such that
\begin{equation}
\mu^{  \delta   } \mu^{    \frac{1 }{p}     } \nu^{  - \frac{1}{p} } \leq   ( \mu^{-1} \nu )^{- 10^{-k } }
\end{equation}
namely, in \eqref{eq:aux_noblowup_22b} we have for any $m\in \NN$
\begin{equation}\label{eq:aux_noblowup_22c}
\begin{aligned}
|\nabla^m u (t)|_{ L^\infty }  \lesssim_{\ep, m,k }  \mu^{ m  -s   }  (\mu \nu^{\frac{1}{2}  } ) ( \mu^{-1} \nu )^{- 10^{-k } }\quad \text{for all}\,\,  0\leq t    \leq t_0 .
\end{aligned}
\end{equation} 
Compared to $|\nabla^m \overline{u}   (t)|_{ L^\infty }$, we have lost a very small power of $ ( \mu^{-1} \nu )$ in \eqref{eq:aux_noblowup_22c}.

It then follows from  \eqref{eq:aux_noblowup_22c}   and the inductive assumption at levels $ m \leq  k-1$ that 
\begin{equation}\label{eq:aux_noblowup_23}
\begin{aligned}
 & \sum_{   m  \leq k -1 }    |\nabla^m  w   |_{L^{2}}  \Big(   |   \nabla^{k +1 - m }     \overline{u}  |_{L^\infty}  +|   \nabla^{k+1 - m }      {u}  |_{L^\infty}    \Big) \\
&  \leq  
C_{\ep, k } \sum_{    m  \leq k -1 }  \mu^{  m   -  s    } (\mu^{-1}  \nu )^{\delta_{ m }  }   \Big(   \mu^{k+1-m-s } (\mu \nu^{\frac{1}{2}  } )           (\mu^{-1}  \nu )^{  - 10^{-k }   }  \Big)  \\
&  \leq  
C_{\ep, k }   \mu^{k - s } (\mu^{2 - s } \nu^{\frac{1}{2}  } ) \sum_{    m  \leq k -1 }  (\mu^{-1}  \nu )^{\delta_{ m }  }                 (\mu^{-1}  \nu )^{  - 10^{-k }   }   .
\end{aligned}
\end{equation}
Using \eqref{eq:aux_noblowup_23},  that  $k\mapsto \delta_k >0$ is decreasing, and $
- 10^{-k }+  \delta_{ k-1  }  \geq \delta_{ k }
$ by the definition of $ \delta_{ k } $, for the term  $J$ we have
\begin{equation}\label{eq:aux_noblowup_23a}
\begin{aligned}
J &  \leq  
C_{\ep, k }  |  \nabla^k w  |_{L^{2}}   \mu^{  k  -  s    }    (\mu^{2-s} \nu^{\frac{1}{2}  } )          (\mu^{-1}  \nu )^{  - 10^{-k  }+  \delta_{ k - 1}    }  \\
& \leq  
C_{\ep, k } |  \nabla^k w  |_{L^{2}}   \mu^{  k  -  s    }    (\mu^{2-s} \nu^{\frac{1}{2}  } )           (\mu^{-1}  \nu )^{     \delta_{ k  }    }  . 
\end{aligned}
\end{equation}

By Young's inequality again,  
\begin{equation}\label{eq:aux_noblowup_21b}
\begin{aligned}
J \lesssim   &     \mu^{2  -s } \nu^{\frac{1}{2}  }           |  \nabla^k w   |_{L^{2}}^2  + (\mu^{2  -s } \nu^{\frac{1}{2}  } )^{2 -1} \big[  \mu^{  k    -  s    }     (\mu^{-1}  \nu )^{ \delta_{k } }\big]^2 \\
 \lesssim &   \mu^{2  -s } \nu^{\frac{1}{2}  }        |  \nabla^k w  |_{L^{2}}^2   +  \mu^{2k -2s } \mu^{2  -s } \nu^{\frac{1}{2}  } (\mu^{-1}  \nu )^{ 2\delta_{k }  }      .
\end{aligned}
\end{equation}

\noindent
\textbf{\underline{Final conclusion}}

Collecting the estimates \eqref{eq:aux_noblowup_I_3} and \eqref{eq:aux_noblowup_21b} of $I,J $,  at level $k$, on the interval $ [0,t_0] $,  we have the  following differential inequality
\begin{equation}\label{eq:aux_noblowup_24}
\begin{aligned}
    \frac{d}{dt} |  \nabla^k  w  |_{L^{2}}^2  & \lesssim_{\ep, k }    \mu^{2  -s } \nu^{\frac{1}{2}  }  |  \nabla^k w  |_{L^{2}}^2         \\
& \qquad +           \mu^{2k - 2s } \big(\mu^{2  -s } \nu^{\frac{1}{2}  } \big) \big[ (  \mu^{-1}  \nu)^{ 2  } +   (\mu^{-1}  \nu )^{ 2 \delta_{k }  }   \big]     .
\end{aligned}
\end{equation}

By Gronwall's inequality  and $0<2 \delta_{k } \leq  2 $, we obtain
\begin{equation}\label{eq:aux_noblowup_25}
\begin{aligned}
 |  \nabla^k  w (t) |_{L^{2}}^2  & \leq C_{\ep , k  }      t^*       \mu^{2k - 2s } \big( \mu^{2  -s } \nu^{\frac{1}{2}  }  \big)    (\mu^{-1}  \nu )^{ 2 \delta_{k } }   \\
& \leq  C_{\ep , k  }           \mu^{2 k -  2 s } (\mu^{-1}  \nu )^{ 2 \delta_{k } } 
\end{aligned}
\quad \text{for all}\,\,  0\leq t    \leq t_0 .
\end{equation}

So the induction is proved at level $k$ as well, namely \eqref{eq:lemma_Hk_w}  holds for all integer $k \in \NN$.

\end{proof}

\subsection{Proof of main proposition}

With all the preparations in hand, we can finish the

\begin{proof}[Proof of Proposition \ref{prop:no_blowup}]

Assume that the bootstrap assumption \eqref{eq:aux_noblowup_claim} is satisfied for some $0 < t_0 <t^*$.

In view of the Sobolev embedding
\begin{equation}\label{eq:aux_noblowup_9}
 | \nabla w    |_{L^\infty }  \lesssim_\delta     |   w  |_{H^{ \frac{ 5}{2}+\delta }  }  \quad \text{for any $\delta>0$}
\end{equation}
 we can use Lemma \ref{lemma:bootstrap_Hk}  to obtain that 
\begin{equation}\label{eq:aux_noblowup_90}
\begin{aligned}
 | \nabla w (t)  |_{L^\infty } & \leq C_{\ep} \mu^{\frac{5}{2} + \delta  -s }     (\mu^{-1} \nu )^{ \frac{9}{10}} \\
&  \leq C_{\ep} \mu^{2  -s }  \nu^{\frac{1}{2}}   \mu^{\delta} (\mu^{-1} \nu )^{ \frac{4}{10}} \quad \text{ for any $t \in [0, t_0]$.}
\end{aligned}
\end{equation}
Thanks to the fact $\nu = \mu^{1- b}$ with $b(s)>0$, in \eqref{eq:aux_noblowup_90} we can choose $\delta(s) >0$ sufficiently small such that
\begin{equation}\label{eq:aux_noblowup_9a}
 | \nabla w (t)  |_{L^\infty } \leq C_{\ep} \mu^{2  -s } \nu^{\frac{1}{2}  } \mu^{-\delta' } \quad \text{ for     some fixed $\delta' >0$.}
\end{equation}

Thanks to the negative exponent on $\mu$ in \eqref{eq:aux_noblowup_9a}, by taking $\mu \geq 1$ sufficiently large depending on $\ep$, we have 
\begin{equation}\label{eq:aux_noblowup_9b}
 | \nabla w (t)  |_{L^\infty } \leq \frac{1 }{2}  M_\ep  \mu^{2  -s } \nu^{\frac{1}{2}  }  \quad \text{for any $t \in [0,t_0]$} 
\end{equation}
where $M_\ep \geq 1$ is the same constant from the bootstrap assumption \eqref{eq:aux_noblowup_claim}.

Combining \eqref{eq:aux_noblowup_claim} and \eqref{eq:aux_noblowup_9b}, for any $t \in [0,t_0]$ we have  
\begin{equation}\label{eq:aux_noblowup_9c}
\begin{aligned}
 | \nabla u (t)  |_{L^\infty } & \leq | \nabla  \overline{u}   (t)  |_{L^\infty } +  | \nabla w (t)  |_{L^\infty } \\
& \leq  \frac{3}{2} M_\ep  \mu^{2  -s } \nu^{\frac{1}{2}  }  .
\end{aligned}
\end{equation}

In other words, we have shown that under the bootstrap assumption \eqref{eq:aux_noblowup_claim} with $0< t_0 < t^*$, the improved bound \eqref{eq:aux_noblowup_9c} holds. From here we conclude that  
\begin{equation}\label{eq:aux_noblowup_9ca}
 | \nabla u (t)  |_{L^\infty } \leq \frac{3}{2} M_\ep  \mu^{2  -s } \nu^{\frac{1}{2}  }   \quad \text{for any $t \in [0,t^*]$} 
\end{equation}
and $u$ does not blowup at $ t = t^*$.

Finally, the bound \eqref{eq:prop_no_blowup} is a direct  consequence of  \eqref{eq:lemma_Hk_w} from Lemma \ref{lemma:bootstrap_Hk} when $t_0 = t^*$.

\end{proof}

\section{Proof of Theorem \ref{thm:Hs_NS}}\label{sec:NS_proof}

In this last section, we outline the minor changes to the construction  in the inviscid case in order to prove Theorem \ref{thm:Hs_NS}.

Given $0< s< \frac{1}{2}$, we adjust the exponent  parameter $b>0$ as follows,
\begin{equation}\label{eq:def_nu_lambda_NS}
\begin{cases}
 b  =  \frac{  \frac{1}{2}  - s }{ 100 } >0 & \\
 \nu  = \mu^{1- b }   . &
\end{cases}
\end{equation}

The adjustment  is to suppress the effect of dissipation. We have a crucial relation in the viscous setting:
\begin{equation}\label{eq:small_dissip_NS}
  \mu^{2} \leq  \mu^{-1} \nu   (\mu^{2 -  s}    \nu^\frac{1}{2})
\end{equation}

Indeed, \eqref{eq:small_dissip_NS} is equivalent to
$
 \mu^{ 1+ s  } \leq  \nu^{ \frac{3}{2}}  
$
which holds due to $ 1+s < (1 - b) \frac{3}{2}$ by \eqref{eq:def_nu_lambda_NS}.

We use the same choice for the critical time $t^* > 0$:
\begin{equation}\label{eq:def_critical_t*_NS}
 t^* =  \ep^{- N  -2 }  \mu^{ - 2  + s } \nu^{- \frac{1}{2} },
\end{equation}
where we still have that $t^* \to 0$ as $\mu \to \infty$. Note that in this case $ t^* \mu^{   2 } \leq C_\ep  \mu^{-1} \nu    $, so the dissipation is negligible in our argument.

We take the same initial data $u_0$ and approximate solution $\overline{u}$ defined by \eqref{eq:def_u_0} and respectively \eqref{eq:def_overline_u} with the new parameters \eqref{eq:def_nu_lambda_NS} in place.

\subsection{Sketch of the proof}
Next, we review the small changes to results in Section \ref{sec:overline_u} and Section \ref{sec:no_blowup_proof} for the viscous case.

As in the inviscid case, on the interval $[0,t^*]$, the vector field $\overline{u}$ satisfies the same estimates as \eqref{eq:prop_error_1}
\begin{align}\label{eq:prop_error_NS_1}
| \overline{u} (t) |_{W^{k,p}   }  \leq     C_{ k,p }  \ep^{2 - k N   }  \mu^{ k -s  } \mu^{1 - \frac{2}{p} } \nu^{ \frac{1}{2}  - \frac{ 1}{p} } .
\end{align}

However, the approximate system \eqref{eq:prop_error_2} in Proposition \ref{prop:approximate} is slightly different:
\begin{equation}\label{eq:lemma_error_NS_2}
\begin{cases}
   \p_t  \overline{u} -\Delta \overline{u} + \overline{u}\cdot \nabla \overline{u}  + \nabla \overline{p} = \overline{E} & \\
\overline{u} |_{t = 0 } = u_0  & .
\end{cases}
\end{equation}
and the  error field $\overline{E}$ has one more term due to the added dissipation: 
\begin{equation}\label{eq:lemma_error_NS_4}
\overline{E}_d : = - \Delta \overline{u}  .
\end{equation} 

We need to show this new error term \eqref{eq:lemma_error_NS_4} is compatible with the estimate \eqref{eq:prop_error_3}. Indeed, by \eqref{eq:prop_error_NS_1}, we have
\begin{equation}\label{eq:lemma_error_NS_5}
| \nabla^k \overline{E}_d (t)|_{L^2 } \leq C_{\ep , k }\mu^{k+2 - s } ,
\end{equation}
and it follows from \eqref{eq:small_dissip_NS} that   the error field $ \overline{E}  $ in the viscous case satisfies the same estimate as  \eqref{eq:prop_error_3}
\begin{equation}\label{eq:lemma_error_NS_3}
\begin{aligned}
    | \nabla^k \overline{E}(t)|_{L^2(\RR^3)} & \leq   C_{\ep,k }   \mu^{k - s } \Big(\mu^{-1} \nu   (\mu^{2 -  s}    \nu^\frac{1}{2} ) + \mu^{2  }  \Big)   \\
& \leq   C_{\ep,k }   \mu^{k - s } \Big(\mu^{-1} \nu   (\mu^{2 -  s}    \nu^\frac{1}{2} )   \Big)   .
\end{aligned}
\end{equation}
In other words, Proposition \ref{prop:approximate} also holds in the viscous case for the system \eqref{eq:lemma_error_NS_2}.

Finally, we review the argument in Section \ref{sec:no_blowup_proof} for the viscous case. In the proof of Proposition \ref{prop:no_blowup}, the key ingredient is the estimates for $ |  \nabla^k  w  |_{L^{2}}^2$, which were proved in Lemma \ref{lemma:bootstrap_L2} and Lemma \ref{lemma:bootstrap_Hk}.

 In the case of 3D Navier-Stokes equation, we have the following system for $w = u  - \overline{ u}$
\begin{equation}\label{eq:aux_noblowup_NS}
\begin{cases}
\p_t w  - \Delta w + u\cdot \nabla w + w \cdot \nabla  \overline{u} + \nabla q = \overline{E} & \\
\D w = 0 & \\
w |_{t = 0 } = 0  . &
\end{cases}
\end{equation}

Since the new error term $\overline{E} $ obeys the same estimates, the  following differential inequality holds for $k = 0 $ as in Lemma \ref{lemma:bootstrap_L2}:
\begin{align}\label{eq:aux_noblowup_NS_2a}
    \frac{d}{dt} |    w  |_{L^{2}}  & \lesssim  C_\ep \mu^{2 -s} \nu^\frac{1}{2}  |   w  |_{L^{2}}        +           \mu^{ -  s } \big(\mu^{  2 -s}  \nu^\frac{1}{2}  \big)  (\mu^{-1}  \nu )  ,  
\end{align} 
where we dropped the positive dissipation term on the left-hand side.

Repeating the argument of Lemma \ref{lemma:bootstrap_Hk} we can also recover \eqref{eq:aux_noblowup_24} for the system \eqref{eq:aux_noblowup_NS}, and obtain for $k \in \NN$ the differential inequality
\begin{align}\label{eq:aux_noblowup_NS_2}
    \frac{d}{dt} |  \nabla^k  w  |_{L^{2}}^2  & \lesssim  C_\ep \mu^{2 -s} \nu^\frac{1}{2}  |  \nabla^k w  |_{L^{2}}^2        +           \mu^{2k - 2s } \big(\mu^{  2 -s}  \nu^\frac{1}{2}  \big)  (\mu^{-1}  \nu )^{ 2 \delta_k } ,  
\end{align} 
where we have also discarded  the  dissipation term on the left.

By \eqref{eq:aux_noblowup_NS_2a} and \eqref{eq:aux_noblowup_NS_2}, Lemma \ref{lemma:bootstrap_L2} and Lemma \ref{lemma:bootstrap_Hk} also hold for the system \eqref{eq:aux_noblowup_NS} in the viscous case. Therefore, we can still prove Proposition \ref{prop:no_blowup}, and Theorem \ref{thm:Hs_NS} follows.

\subsection*{Acknowledgment}

The author is grateful to the organizers of the workshop ``Analysis of fluids'' at Osaka Metropolitan University, Osaka, Japan in March 2024 for their warm hospitality, where part of the work was done. He is also grateful to Tarek Elgindi for stimulating conversations. Special thanks to   Alexey Cheskidov for pointing out an  exponent error in a previous version of the manuscript.

\appendix

\bibliographystyle{alpha}
\bibliography{euler_ill_2024}

\end{document}